\documentclass[reqno,11pt]{amsart}

\usepackage{amsmath, amsthm, a4, latexsym, amssymb, pstricks, setspace, mathtools}
\usepackage{indentfirst}
\usepackage[dvips]{graphicx}
\usepackage{bbm}
\usepackage{mathrsfs}

\usepackage{titlesec}
 
\titleformat{\section}[runin]
{\normalfont\bfseries}
{\thesection.}{.5em}{}[.]
\titlespacing{\section}
{\parindent}{1.5ex plus .1ex minus .2ex}{\wordsep}

\def\@rmrk#1#2{\refstepcounter
    {#1}\@ifnextchar[{\@yrmrk{#1}{#2}}{\@xrmrk{#1}{#2}}}

\makeatletter\@addtoreset{equation}{section}\makeatother

 \newfont{\bfit}{cmbxti10 scaled 1200}

 \newcommand{\eps}{\varepsilon}

 \newcommand{\R}{\mathbb{R}}
 \newcommand{\N}{\mathbb{N}}
 \newcommand{\Z}{\mathbb{Z}}
 
 \newcommand{\prob}{\mathbb{P}}

 \newcommand{\1}{{\sf 1}}

 \newcommand{\sfrac}[2]{\mbox{$\frac{#1}{#2}$}}
 
 \newcommand{\ssup}[1] {{\scriptscriptstyle{({#1}})}}

\renewcommand{\subsection}{\secdef \subsct\sbsect}
\newcommand{\subsct}[2][default]{\refstepcounter{subsection}
\vspace{0.15cm}
{\flushleft\bf \arabic{section}.\arabic{subsection}~\bf #1  }
\nopagebreak\nopagebreak}
\newcommand{\sbsect}[1]{\vspace{0.1cm}\noindent
{\bf #1}\vspace{0.1cm}}

{\nopagebreak {\hfill\rule{2mm}{2mm}}
\\ }

\newtheorem{theorem}{Theorem}[section]
\newtheorem{thm}{Theorem}
\newtheorem{lemma}[theorem]{Lemma}
\newtheorem{corollary}[theorem]{Corollary}
\newtheorem{prop}[theorem]{Proposition}

\newcommand{\p}{\mathbb{P}}

\newcommand{\E}{\mathbb{E}}

\newtheoremstyle{thm}{1.5ex}{1.5ex}{\itshape\rmfamily}{}
{\bfseries\rmfamily}{}{2ex}{}

\newtheoremstyle{rem}{1.3ex}{1.3ex}{\rmfamily}{}
{\itshape\rmfamily}{}{1.5ex}{}
\theoremstyle{rem}
\newtheorem{remark}{{\slshape\sffamily Remark}}[]

\refstepcounter{subsubsection}

\def\thebibliography#1{\section*{References}
  \list%
  {\arabic{enumi}.}
    {\settowidth\labelwidth{[#1]}\leftmargin\labelwidth
    \advance\leftmargin\labelsep
    \parsep0pt\itemsep0pt
    \usecounter{enumi}}
    \def\newblock{\hskip .11em plus .33em minus .07em}
    \sloppy                   
    \sfcode`\.=1000\relax}

\begin{document}

\title{Optimal embeddings by unbiased shifts 
of Brownian motion}
\author{Peter M\"orters and Istv\'an Redl}

\maketitle
\ \\[-2cm]

\begin{abstract}
An unbiased shift of the two-sided Brownian motion $(B_t \colon t\in\R)$ is a random time $T$ such that $(B_{T+t} \colon t\in\R)$ is still a two-sided Brownian motion.  Given a pair $\mu, \nu$ of orthogonal probability measures, an unbiased shift $T$ solves the embedding problem, if $B_0\sim\mu$ implies $B_{T}\sim\nu$. 
A solution to this problem was given by Last  et al.~(2014), based on earlier work of Bertoin and Le Jan (1992), and Holroyd and Liggett (2001). 
In this note we show that this solution minimises $\E \psi(T)$ over all nonnegative unbiased solutions~$T$,  simultaneously for all nonnegative, concave functions~$\psi$. Our proof is based on a discrete concavity inequality that may be of independent interest.
\end{abstract}
\ \\[-5mm]

{\footnotesize
}\\[-2mm]

\section{Introduction}\label{S:Intro}

A random time $T$ is called an \emph{unbiased shift} of the two sided Brownian motion $(B_{t} \colon t \in \R)$ if the shifted process $(B_{T+t} \colon t\in \R)$ is again a two-sided Brownian motion. This notion was introduced and studied by Last et al.\ in~\cite{LMT14}, where the additional requirement of measurability
of~$T$ with respect to the process $(B_{t} \colon t \in \R)$ is made, which we drop here for greater generality. The concept of unbiased shifts goes back to a similar idea for coin tosses, known as the extra head problem in~\cite{HL01, LT02, HP05}, and, more generally, the concept of {shift-couplings}, see the work of Thorisson~\cite{Th00}.%
\medskip

In~\cite{LMT14} the authors solve the \emph{embedding problem} for unbiased shifts. Namely, given two orthogonal probability measures $\mu, \nu$ on the real line such that $B_0 \sim \mu$ they construct a nonnegative unbiased shift~$T_*$ such that  $B_{T_*} \sim \nu$. The solution of~\cite{LMT14}  can easily be described explicitly. Let $(L^x_t \colon x\in\R, t\geq 0)$ be a  continuous version of the local time for $(B_{t} \colon t \geq0)$. We use this to build two continuous additive functionals by letting
$$L^\nu_t:=\int L^x_t \, \nu(dx), \mbox{ and } L^\mu_t:=\int L^x_t \, \mu(dx),$$
and obtain the solution as
\begin{equation}\label{tstar}
T_*:=\inf\{ t>0 \colon L^\nu_t =L^\mu_t\}.\phantom{_{M_{_M}}}
\end{equation}
Note that $T_*$ occurred in the context of one-sided (Skorokhod) embedding problems in the work of
Bertoin and Le Jan~\cite{BL92}, is reminiscent of extra head schemes in~\cite{LT02, HP05} or~\cite{MR15}, and turns out to be closely related to allocation and transport problems, see~\cite{HPPS09, BH13}. 

The present paper is concerned with the problem whether this natural and explicit solution of the embedding
problem is optimal in the sense that it minimises certain moments. To see which moments should be looked at, Last et al.~\cite{LMT14} have investigated finiteness of moments for unbiased shifts,  and have shown that, for any unbiased shift~$T$ embedding a measure orthogonal to the initial distribution,  the random variable  $|T|$ has infinite square-root moment. Under the additional assumption that $T$ is a nonnegative stopping time they obtain that $T$ has an infinite fourth-root moment.
It is conjectured in~\cite{LMT14} that this holds without the assumption that $T$ is a stopping time, and 
this conjecture will be confirmed in this paper, see Remark~1(c) below.  The results of~\cite{LMT14} show that in order to understand optimality we need to focus on moments of fractional order smaller than $\frac12$, or more generally moments taken with respect to concave gauge functions.  \smallskip

We denote by $\prob_\mu, \E_\mu$ probability and expectation on a probability space supporting
a two-sided Brownian motion $(B_{t} \colon t \in \R)$ with $B_0\sim\mu$.
We can now state our main result.\medskip
 
\begin{thm}\label{T:Brownian_optimality}
Assume that $\mu$ and $\nu$ are two orthogonal probability measures on $\R$.  
For any non-negative  unbiased shift $T$ embedding $\nu$ and all $\psi\colon  [0,\infty) \to [0,\infty)$ concave,
\[
\E_{\mu} \psi(T_{*})\le \E_{\mu}\psi(T), \phantom{_{_{_g}}}
\]
where $T_{*}$ is the unbiased shift constructed in \eqref{tstar}.
\end{thm}

\begin{remark}\ \\[-5mm]
\begin{itemize}
\item[(a)] This result is closely related to a similar optimality result in~\cite{MR15} for the case of discrete-time Markov chains. Although our proof relies on a discrete approximation, we have been unable to derive Theorem~\ref{T:Brownian_optimality} directly from the results of~\cite{MR15}. Instead, we use a concavity inequality, stated as Lemma~\ref{3rdconcavity}$(b)$ below, that may be of independent interest.\\[-3mm]

\item[(b)]  We make no assumptions on measurability of $T$ with respect to the Brownian motion, i.e., $T$ is allowed to use additional randomisation beyond that taken from the Brownian motion.\\[-3mm]

\item[(c)]  As $T_*$ is a stopping time, we have $\E_{\mu} T_{*}^{_{1/4}}=\infty$ by Theorem~8.1 in~\cite{LMT14},
and hence $$\E_{\mu} T_{\phantom{*}}^{_{1/4}}=\infty$$ for \emph{all} non-negative solutions~$T$ of the embedding problem.\\[-3mm]

\item[(d)]  Under mild conditions on $\mu, \nu$ we have $\E_{\mu} T_{*}^\alpha<\infty$ for $\alpha<\frac14$
by Theorem~8.2 in~\cite{LMT14}. \\[-3mm]
\item[(e)] The result is strongly reminiscent of optimality results for Skorokhod embeddings, as given, for example, in the classical papers~\cite{CH04, CH07}. The use of optimal transport ideas in connection to optimal Skorokhod embeddings is also the topic of a lot of current research,  of which the paper~\cite{BH13} is a major highlight.\\[-2mm]
\end{itemize}

\end{remark}

\begin{remark}
If $\mu$ and $\nu$ are \emph{not} orthogonal we can write 
\smash{$\mu=\tilde\mu+\mu \wedge \nu$} and
\smash{$\nu=\tilde\nu+\mu \wedge \nu$} for orthogonal sub-probability measures \smash{$\tilde\mu, \tilde\nu$}. Denoting the total mass of both measures by $\rho$ one can enlarge the probability space to include a Bernoulli variable $U$, depending only on $B_0$, with \smash{$\p_\mu(U=1)=\rho$} such that on the event $U=1$ we have 
\smash{$B_0 \sim \frac1\rho \, \tilde\mu$} and on the event $U=0$ we have 
\smash{$B_0\sim \frac1{1-\rho}\, (\mu \wedge \nu)$}. Then 
$$T_*=U \inf\{ t>0 \colon L^{\tilde\nu}_t =L^{\tilde\mu}_t\}$$
is a non-negative unbiased shift embedding $\nu$, which satisfies the inequality in Theorem~\ref{T:Brownian_optimality}. This follows by splitting the expectations according to the value of $U$, and using monotonicity of $\psi$ on the event $U=0$, and 
applying the theorem to the orthogonal probability measures $\frac1\rho \tilde\mu$, $\frac 1\rho \tilde\nu$
on the event $U=1$.\end{remark}

\section{Preliminaries and outline of the proof}\label{S:Prelim}

We recall the framework of~\cite{LMT14}, which will be used as a main reference 
throughout this paper. Let  $(\Omega, \mathscr{A},\p_{0})$ be such that
$\Omega$ is the space of continuous functions $\omega\colon \R \to \R$, equipped with the 
Borel $\sigma$-algebra~$\mathscr{A}$ and the distribution  $\p_{0}$ of two-sided Brownian motion with $\omega_0=0$.   For all $x \in \R$
we define $\p_{x}$ to be the distribution of $\omega+x$  and introduce the $\sigma$-finite measure 
$\p:= \int \p_{x} \, dx$, which is invariant under the shifts $\omega\mapsto s\omega$ 
defined by $(s\omega)(\cdot)=\omega(\cdot-s)$, for all $s\in\R$. As usual, expectations with respect to $\p, \p_{x}$ will be denoted by $\E, \E_{x}$, respectively. 
\medskip

\pagebreak[3]

Let $(\ell^{x} \colon x\in \R)$ be a continuous version of the family of local times of $\omega$ at level $x$, and 
for any locally finite measure $\nu$ on $\R$ set $\ell^{\nu}:=\int \ell^{x}\,\nu(dx)$. We note that the local times $\ell^{x}$ and $\ell^{\nu}$ are random measures and depend on $\omega \in \Omega$.     
When necessary, this dependence will be indicated as $\ell^{\cdot}_{\omega}$. We may and shall assume that
$\ell^{x}_{z+\omega}= \ell^{x-z}_{\omega}$, for all $x,z\in\R$, for $\p$-almost every~$\omega$.
\pagebreak[3]

A \emph{transport rule} is a Markov kernel $\theta\colon \Omega \times \R \times \mathscr{B} \to [0,1]$, where $\mathscr{B}$ is the Borel \mbox{$\sigma$-algebra} on $\R$. We use the notation
$\theta_{\omega}(ds,dt)=\theta_{\omega}(s,dt)\ell^{\mu}(ds)$,  and $\theta_{\omega}(A,B)=\int_{A}\theta_{\omega}(s,B)\ell^{\mu}(ds)$ for any $A, B \in \mathscr{B}$. 
An interpretation of the Markov kernels $\theta(s,\cdot)$ is that each site $s \in \R$
sends out unit mass to the real line. A transport rule $\theta$ is called \emph{translation invariant} if $\theta_{\omega}(s,A)=\theta_{t\omega}(s-t,A-t)$,
for any $(s,A) \in \R \times \mathscr{B}$, $t\in\R$, for $\p$-almost every $\omega$.

We now briefly summarise results of \cite{LMT14}, which are derived from the abstract results of~\cite{LT09}. Given a translation invariant transport rule $\theta$ we obtain an unbiased shift $T$ by letting
\begin{equation}\label{pmu}
\prob_\mu(\omega\in A, T \in B) = \iint_A \theta_\omega(0, B) \,\p_x(d\omega)\, \mu(dx), \mbox{ for } A\in\mathscr A, B\in\mathscr B.
\end{equation}
Conversely, given an unbiased shift $T$ we can construct a  transport rule $\theta$ by letting
$\theta_\omega(s, B)= 
\prob_\mu(T \in B-s \, | \, s\omega )$, for $s\in\R, B\in\mathscr{B}, \omega\in\Omega$.
We use a suitably regularised  version of conditional probabilities so that $\theta$ is a
translation invariant transport rule and~\eqref{pmu} holds.

The transport rules associated in this way with nonnegative unbiased shifts are \emph{forward looking} 
in the sense that $\theta_{\omega}(s,(-\infty,s))=0$ for all $s\in \R$. An unbiased shift $T$ solves the embedding problem for a pair of orthogonal probability measures $\mu, \nu$ if and only if the associated
transport rule satisfies the \emph{balancing property} 
\[
\int\theta_{\omega}(s, dt) \,  \ell_{\omega}^{\mu}(ds) 
=\ell_{\omega}^{\nu}(dt) \qquad \mbox{ for $\p$-almost all~$\omega$}.
\]
If the unbiased shift is not randomised, i.e. a measurable function of~$\omega$, 
then the associated $\theta$ is an \emph{allocation rule}, i.e. it is of the form $\theta_{\omega}(s, A)= \1_A(\tau_{\omega}(s))$
for some $\tau_\omega\colon\R\to\R$. In this case each site $s \in \R$ is assigned to a new site $\tau(s) \in \R$, where again we drop the dependence on $\omega$ from the notation.
The allocation rule associated to the random times $T_*$ is given by
\begin{equation}\label{E:tau_star}
\tau_{*}(s):=\inf\{t>s \colon \ell^{\mu}[s,t]=\ell^{\nu}[s,t]\}, \qquad \mbox{ for all $s \in \R$.}
\end{equation}

Let us now outline the proof of Theorem~\ref{T:Brownian_optimality}. The first part looks at what happens \emph{pathwise}. With every transport rule we associate a local cost. Given a fixed $\omega$, we show that on carefully chosen intervals called excursions, the best possible cost is offered by the allocation rule~$\tau_{*}$. We do this by proving an analogous  discrete result and then taking a suitable limit. The second part of the proof uses  \emph{ergodic theory} to translate the local cost optimality of $\tau_{*}$ into a result on the moments of~$T_*$.

\section{Pathwise level}\label{S:pathwise}
In what follows, we fix a path $\omega$. An \emph{excursion} $\mathcal{E}$ is any interval 
of the form $[a,\tau_*(a)]$, for $a \in \R$. Observe that $\tau_*$ maps $\mathcal{E}$ onto itself.
Our main goal is to show that the allocation rule $\tau_{*}$ offers 
the best possible cost inside an arbitrary excursion $\mathcal{E}$.
More precisely we will devote this section to proving the following proposition.  

\begin{prop}\label{L:pathwise}
Given an excursion $\mathcal{E}$, for any forward-looking 
transport rule $\theta$ balancing $\ell^{\mu}$ and~$\ell^{\nu}$,
\begin{equation}\label{E:cost_pathwise}
\iint_{\mathcal{E} \times \R} \psi(t-s)\,\theta(ds,dt)+\iint_{\R \times \mathcal{E}}\psi(t-s)\,\theta(ds,dt) \ge 2\int_{\mathcal{E}}\psi(\tau_{*}(s)-s)\ell^{\mu}(ds),
\end{equation}
for all $\psi\colon [0,\infty) \to [0,\infty)$ concave.
\end{prop}

The left hand side in \eqref{E:cost_pathwise} is called the \emph{cost} of the transport rule
$\theta$ over the interval~$\mathcal{E}$. Note that $\tau_*(\mathcal{E})=\mathcal{E}$ and hence
the right hand side is then the cost of the 
allocation rule given by~$\tau_*$ over the same interval. By subtracting a constant from both sides of the equation we may henceforth assume that $\psi(0)=0$.
\pagebreak[3]

We start by establishing a similar result in a discrete setting. The inequality we establish below 
is of a general nature and may be of independent interest. A map $\tau\colon A \to B$ between two discrete and disjoint sets $A, B \subset \R$ is the \emph{stable allocation map} if 
$$\tau(a)=\min\big\{b\in B \colon b>a, |B\cap[a,b]| =  |A\cap[a,b]|\big\}.$$

\begin{lemma}\label{3rdconcavity}
Let $a_1>a_2>a_3>\ldots$ and $b_1<b_2<b_3<\ldots$ be disjoint real sequences, such that 
we have $a_n \searrow -\infty$ and $b_n \nearrow \infty$. 
\begin{itemize}
\item[(a)] The stable allocation map
$\tau\colon  \{\ldots, a_3, a_2, a_1\} \to \{b_1, b_2, b_3, \ldots\}$ is well-defined and there exists $N\in\mathbb N$ 
such that $\tau(a_i)=b_i$ for all $i\geq N$.
\item[(b)]
For every concave function $\psi\colon [0,\infty) \to [0,\infty)$ and nonnegative matrix
$\pi = (\pi_{i,j} \colon i,j\in{\mathbb N} )$ with the properties that\\[-2mm]
\begin{itemize}
\item$\pi_{i,j}=0$ if $a_i>b_j$, \\[-2mm]
\item $\sum_{j=1}^\infty \pi_{i,j}=1$ for all $i\in\{1,\ldots,N\}$,\\[-2mm]
\item $\sum_{i=1}^\infty \pi_{i,j}=1$ for all $j\in\{1,\ldots,N\}$,\\[-2mm]
\end{itemize}
we have
\begin{equation}\label{ineq}
\sum_{i=1}^{N}\sum_{j=1}^{\infty} 
+ \sum_{i=1}^{\infty}\sum_{j=1}^{N}\pi_{i,j} \psi(b_{j}-a_{i}) 
\ge 2 \sum_{i=1}^{N}\psi(\tau(a_i)-a_{i}).
\end{equation}
\end{itemize}
\end{lemma}

Note that the sum on the left hand side of \eqref{ineq} is over pairs of positive integers
$i,j$ with $i\wedge j \leq N$, counting the contribution of pairs with $i\vee j \leq N$ twice.
In the special case that $a_1<b_1$ we have $\tau(a_i)=b_i$ for all $i$, and hence $N\in\N$ can be chosen arbitrarily. Then our result becomes the following general result, which to the best of our knowledge is new, too.

\begin{corollary}
For every double-sided increasing sequence $(a_n \colon n \in \Z)$ that is unbounded from above and below, 
every stochastic matrix $\pi = (\pi_{i,j} \colon i,j\in{\mathbb N} )$, and every
concave function  $\psi\colon [0,\infty) \to [0,\infty)$ we have
$$\sum_{i=1}^{n}\sum_{j=1}^{\infty} 
+ \sum_{i=1}^{\infty}\sum_{j=1}^{n} \pi_{i,j} \psi(a_{i}-a_{-j}) 
\ge 2\sum_{i=1}^{n}\psi(a_{i}-a_{-i}).$$
\end{corollary}

\begin{proof}[Proof of Lemma~\ref{3rdconcavity}(a)]
Given $a\in A$ the set $A \cap [a,\infty)$ is finite, but the set
$B \cap [a,\infty)$ is infinite. Hence, on the one hand, there exists $b\in B$ such that
$|B \cap [a,b]|\geq |A \cap [a,\infty)| \geq |A \cap [a,b]|$, while on the other hand
$|B \cap [a,a]|=0<1=|A \cap [a,a]|$. This implies that there exists
$b'\in[a,b]$ with $|B \cap [a,b']|= |A \cap [a,b']|$, and hence $\tau$ is well-defined.
\medskip

We now show that there exists $N \in \N$ such that, for all $n\geq N$, we have $\tau(a_n)=b_n$.
If $a_1<b_1$ one can choose  $N=1$. Otherwise define an integer-valued function
$f\colon[b_1,\infty)\to\R$ by
$$ 
f(x)= \big| A\cap [b_1,x] \big| -\big| B\cap [b_1,x] \big|.
$$
Let $M\in\Z$ be the minimum of $f$ on $[b_1,a_1]$.
Note that $f(a_1)>M$ and on $[a_1,\infty)$ the function $f$ is decreasing to $-\infty$ by downward jumps of size one. Hence there exists  $n> 1$ with $f(b_n)=M-1$.  Clearly, for all $m\geq n$ we have 
$|A\cap[a_m,b_m]|=m=|B\cap[a_m,b_m]|$ while, for $j<m$, 
we have $|A\cap[a_m,b_j]|> j=|B\cap[a_m,b_j]|$. Hence, $\tau(a_m)=b_m$, as required.
 \end{proof}

\begin{proof}[Proof of Lemma~\ref{3rdconcavity}(b)] 
This is a variant of Lemma 5.3 in~\cite{MR15}. We say that four points
$a_k<a_i<b_j<b_l$ are crossed by $\pi$ if 
$\pi_{kj}>0$ and $\pi_{il}>0$. 
Such a crossing can be repaired by replacing the matrix $\pi$ by a new matrix $\pi'$ given by
$$\begin{aligned}
\pi'_{kj} & =  \pi_{kj}- (\pi_{kj} \wedge \pi_{il}), &
\pi'_{il}  & =  \pi_{il}- (\pi_{kj} \wedge \pi_{il}), \\
\pi'_{ij} & =  \pi_{ij} + (\pi_{kj} \wedge \pi_{il}), &
\pi'_{kl}  & =  \pi_{kl} + (\pi_{kj} \wedge \pi_{il}), \\
\end{aligned}$$
leaving all other entries untouched. If $\pi$ satisfies the conditions of $(b)$, then so does $\pi'$.
By concavity of the function $\psi$ we get
$$\psi(b_j-a_k)+\psi(b_l-a_i) \geq \psi(b_j-a_i)+\psi(b_l-a_k).$$
Hence the left hand side of \eqref{ineq} decreases when we replace $\pi$ by $\pi'$, and we observe that
$a_k<a_i<b_j<b_l$ are not crossed by $\pi'$. We say the crossing is repaired.

We systematically repair all crossings as in Section~5 of~\cite{MR15}, i.e.\ in the following order
\begin{itemize}
\item picking $b_j$ from $\{b_1,\ldots, b_n\}$ from left to right,
\item given $b_j$ picking $a_i<b_j$ from $\{a_i \colon 1\leq i \leq n, a_i<b_j\}$ from right to left,
\item given $a_i<b_j$ picking $a_k<a_i$ from $\{a_{i+1}, a_{i+2}, \ldots\}$ from right to left, 
\item given $a_k<a_i<b_j$ picking $b_l$ from $\{b_{j+1}, b_{j+2},\ldots\}$ from left to right.
\end{itemize}
Just as in Lemma~5.1 of~\cite{MR15} we see that this procedure is well-defined (taking limits in the last two steps) and leads to a matrix $\pi^*$, which satisfies
$$ \sum_{i=1}^{N}\sum_{j=1}^{\infty} 
+ \sum_{i=1}^{\infty}\sum_{j=1}^{N}\pi_{i,j} \psi(b_{j}-a_{i}) 
\ge \sum_{i=1}^{N}\sum_{j=1}^{\infty} 
+ \sum_{i=1}^{\infty}\sum_{j=1}^{N}\pi^*_{i,j} \psi(b_{j}-a_{i}) .$$
Moreover, $\pi^*$ crosses no four points $a_k<a_i<b_j<b_l$ with $a_i, b_j\in\mathcal E$ and hence,
as in Lemma~5.2 of~\cite{MR15}, we get that $\pi^*_{ij}=\mathbbm{1}\{\tau(a_i)=b_j\}$.
Plugging the entries of $\pi^*$ into the right hand side gives~\eqref{ineq}.
\end{proof}

We now use Lemma~\ref{3rdconcavity} to get a continuous inequality. Let $\mathcal E=[b_0,a_0]\subset \R$ be an excursion
and $M=\ell^\mu(\mathcal E)>0$. Given $n\in\N$ we pick $a_1> a_2> \ldots>a_n$ such that 
\smash{$\ell^\mu(a_{i}, a_{i-1})=\frac{M}n$} and
$b_1< b_2<\ldots < b_n$  such that $\ell^{\nu}(b_{j-1},b_{j})=\frac{M}n$, for $1\leq i,j \leq n$, and
also such that $a_n=b_0$ and $b_n=a_0$. \smallskip

Additionally, $a_i, b_i$, for 
$i > n$, are chosen in such a way that $a_i \searrow -\infty$, $b_i \nearrow \infty$, and
$$
\sup_{i=n}^\infty \big( b_{i+1}-b_i\big),   \quad\sup_{j=n}^\infty \big( a_j-a_{j+1}\big)\longrightarrow 0.
$$
Then, if $\tau_n \colon \{\ldots, a_3, a_2, a_1\} \to \{b_1, b_2, b_3, \ldots\}$ is the (discrete) stable allocation map, in Lemma~\ref{3rdconcavity} we may choose $N=n$.
Now suppose a forward-looking and balancing transport rule $\theta$ is given. We define
$$\pi_{i,j}= \frac{n}M\, \theta((a_{i},a_{i-1}], (b_{j-1},b_j])$$
and note that $\pi = (\pi_{i,j} \colon i,j\in{\mathbb N} )$ satisfies the conditions
of Lemma~\ref{3rdconcavity}$\,(b)$. Hence we have
\begin{equation}\label{E:discrete_ineq_conv}
\sum_{i=1}^{n}\sum_{j=1}^{\infty} 
+ \sum_{i=1}^{\infty} \sum_{j=1}^{n}\pi_{i,j} \psi(b_{j}-a_{i}) 
\ge 2\sum_{i=1}^{n}\psi(\tau_n(a_i)-a_{i}).
\end{equation}
Multiply both sides by $M/n$ and let $n$ go to infinity. We will argue that the right and left hand side
of~\eqref{E:discrete_ineq_conv}  converge to those of~\eqref{E:cost_pathwise}.
\smallskip

\emph{First} we show convergence of the left hand side of~\eqref{E:discrete_ineq_conv}. Given $\{a_i\}_{i\in \N}$ and $\{b_{j}\}_{j \in \N}$ as constructed above, let $g_{n}(a)=a_i$, if $a \in (a_i, a_{i-1}]$ 
and $h_n(b)=b_j$, if $b \in (b_{j-1}, b_j]$.

\begin{lemma}\label{L:g_h_convergences}
For $\ell^{\mu}$-almost every~$a$ we have $g_n(a) \to a$, and for 
$\ell^{\nu}$-almost every~$b$ we have $h_n(b) \to b$.
\end{lemma}

\begin{proof}
It suffices to prove the first claim. The result is trivial if $a\leq b_0$.
Given $\varepsilon>0$, for $\ell^{\mu}$-almost every ~$a$ with
$b_0+\eps<a \leq a_0$ there exists $\eta>0$ such that
$\ell^{\mu}(a-\varepsilon, a) \ge \eta$. Hence, for all $n > M/ \eta$, 
there exists $a_i \in (a-\varepsilon, a)$ which implies $0 < a-g_n(a) < \varepsilon$.  
\end{proof}

Note that
$$\frac M n \sum_{i=1}^{n}\sum_{j=1}^{\infty} \pi_{i,j} \psi(b_{j}-a_{i}) 
= \iint_{\mathcal E \times \R}  \psi\big(h_n(b)-g_n(a)\big)\theta(da,db).$$
Now we compare the integrand with $\psi\big(b-a)$. Adding and subtracting $\psi(b-g_n(a))$ and then using the triangle inequality we get 
\begin{align}
|\psi(h_n(b)-g_n(a))-\psi(b-a)|& \le  |\psi(h_n(b)-g_n(a))-\psi(b-g_n(a))|+
|\psi(b-a)-\psi(b-g_n(a))| \notag \\ 
& \le  \psi(h_n(b)-b)+\psi(a-g_n(a)), \notag
\end{align}
where in the second inequality we used the sub-additivity of $\psi$ as follows, 
\begin{align}
|\psi(h_n(b)-g_n(a))-\psi(b-g_n(a))| &= \psi(h_n(b)-g_n(a))-\psi(b-g_n(a))\notag\\
&=
\psi(h_n(b)-b+b-g_n(a))-\psi(b-g_n(a)) \le \psi(h_n(b)-b). \notag
\end{align}
Using this estimate on the integrand, we get
\begin{align*}
0 \leq \iint_{\mathcal E \times \R}  & \psi\big(h_n(b)-g_n(a)\big)-\psi(b-a) \, \theta(da, db)  \leq \iint_{\mathcal E \times \R} \psi(h_n(b)-b)+\psi(a-g_n(a)) 
\, \theta(da, db).
\end{align*}
The integrand on the right is bounded and converges to zero, $\theta$-almost everywhere, by Lemma~\ref{L:g_h_convergences}.  Hence the left hand side 
of~\eqref{E:discrete_ineq_conv}, multiplied by $\frac M n$, converges to the required limit, 
$$\iint_{\mathcal E \times \R}  +   \iint_{\R \times \mathcal E} \psi(b-a)\, \theta(da,db).$$   

\emph{Second} we show convergence of the right hand side of~\eqref{E:discrete_ineq_conv}. The key to this is the following lemma.

\begin{lemma}\label{L:tau_n convergence}
For $\ell^{\mu}$-almost every $a \in \mathcal{E}$, we have $\lim_{n\to \infty} \tau_{n}(g_{n}(a))=\tau_*(a)$.
\end{lemma}

\begin{proof}
We define
$$f\colon [a, \tau_*(a)] \to [0,\infty), f(x)= \ell^\mu[a,x]-\ell^\nu[a,x],$$
and
$$f_n\colon [a, \tau_*(a)] \to \R, f_n(x)= \sfrac{M}n \big( |\{ i \colon a_i\in[g_n(a),x]\}| - |\{ j \colon b_j\in[g_n(a),x]\}| \big).$$
The proof is organised into five steps.
\pagebreak[3]

{\bf Step 1:} $|f(x)-f_n(x)| \leq \frac{4M}{n}$.

\begin{proof} Denote $k_1=|\{ i \colon a_i\in[g_n(a),x]\}|$ and $k_2=|\{ j \colon b_j\in[g_n(a),x]\}|$, then
$f_n(x)= \frac{M}n (k_1-k_2)$ and
\begin{align*}
\ell^\mu[a,x]& = (k_1-1) \frac{M}n+ \ell^\mu[g_n(x),x]-\ell^\mu[g_n(a),a],
\end{align*}
and hence
$|\ell^\mu[a,x]- k_1 \frac{M}n| \leq \frac{2M}{n}.$
Similarly,
$|\ell^\nu[a,x]- k_2 \frac{M}n| \leq \frac{2M}{n},$
which implies the statement.
\end{proof}

Recall that $\tau_*(a)=\inf\{x>a \colon f(x)=0\}$ and
$$\tau_n(g_{n}(a))=\inf\{x>g_n(a) \colon f_n(x)=0\}.$$

{\bf Step 2:}
For $\ell^{\mu}$-almost every $a \in \mathcal{E}$, there exists $\eps_0>0$ such that, 
$$f(a+x) \geq \sfrac56 \ell^\mu[a,a+x]  \mbox{ for all } x\in(0,\eps_0), \mbox{ and }$$
$$f(\tau_*(a)-x) \geq \sfrac56 \ell^\nu[\tau_*(a)-x,\tau_*(a)]  \mbox{ for all } x\in(0,\eps_0).$$

\begin{proof}
For $\ell^{\mu}$-almost every $a \in \mathcal{E}$, there exists $\eps_0>0$ such that, 
$$\ell^\nu[a,a+x] \leq \sfrac16 \ell^\mu[a,a+x]  \mbox{ for all } x\in(0,\eps_0).$$
see, e.g., Section~1.6, Theorem~3 in \cite{EG}. 
This implies the first property, and the second is analogous.\end{proof}

{\bf Step 3:} 
Let 
$G_n := \{ a_i \colon i\in\{1,\ldots,n\} ,  \not\!\!\exists b_j \in [a_i, a_{i-7}] \}.$
Then, 
for $\ell^\mu$-almost every $a\in \mathcal E$, there exists arbitrarily large $n$ with $g_n(a)\in G_n$.

\begin{proof}
Let $\eta>0$. 
For $\ell^\mu$-almost every $x\in \mathcal E$, 
there exists $\delta_0>0$ such that $[x-\delta_0, x+\delta_0]\subset\mathcal E$ and
$$\ell^\nu[x-\delta, x+\delta] \leq \eta \ell^\mu[x-\delta, x+\delta], \quad
\mbox{ for all }0<\delta<\delta_0,$$
see, e.g., Section~1.6, Theorem~3 in \cite{EG}. 
Supposing that, for some $k\in\N$,
$$ \sfrac{(k+2)M}{n} \leq \ell^\mu[x-\delta, x+\delta] \leq \sfrac{(k+3)M}{n}$$
we have 
$$\big|\big\{ a_i \colon i\in\{1,\ldots,n\} ,  a_i\in[x-\delta, x+\delta]  \big\} \big| \geq k+1$$
and
$$\big|\big\{ b_j \colon j\in\{1,\ldots,n\} ,  b_j\in[x-\delta, x+\delta]  \big\} \big| \leq 1+ \eta(k+3).$$
By the pigeonhole principle therefore
$$\big| \big\{ a_i\in[x-\delta, x+\delta] \colon  a_i \not\in G_n   \big\} \big| 
\leq 7 ( 1+ \eta(k+3)). $$
Hence, given $\eps>0$, we can find $\eta$, and hence $\delta_0(x)>0$, such that
$$\ell^\mu\{ a\in[x-\delta, x+\delta] \colon g_n(a) \not\in G_n\} <\sfrac\eps{4M}\, \ell^\mu[x-\delta, x+\delta],\quad
\mbox{ for all }0<\delta<\delta_0(x) \mbox{ and } n \geq n_0(\delta,x).$$  
Let $$\mathcal E(\delta):=\{ x\in \mathcal E \colon \delta_0(x)>\delta\}.$$ \ \\[-2mm]
We now fix a global $\delta>0$ with $\ell^\mu(\mathcal E \setminus \mathcal E(\delta))<\frac\eps2$.
We then cover $\mathcal E(\delta)$ by finitely many intervals $[x_1-\delta, x_1+\delta], \ldots,
[x_m-\delta, x_m+\delta]$ with $x_i\in\mathcal E(\delta)$ such that each  of the centres $x_1,\ldots, x_m$ is only contained in one of the intervals. Let $n_0:=\max_{i=1}^m n_0(\delta,x_i)$. Then, for all $n\geq n_0$, we have
\begin{align*}
\ell^\mu\{ a\in \mathcal E \colon g_n(a) \not\in G_n \} &
\leq \sum_{i=1}^m \ell^\mu\{ a\in[x_i-\delta, x_i+\delta] \colon g_n(a) \not\in G_n\} + 
\ell^\mu(\mathcal E \setminus \mathcal E(\delta))\\
& \leq \frac\eps{4M}\, \sum_{i=1}^m \ell^\mu[x_i-\delta, x_i+\delta] + \frac\eps2
 \leq \frac\eps{4M} \,  2\ell^\mu(\mathcal E) + \frac\eps2 = \eps.
\end{align*}
The result follows as $\eps>0$ was arbitrary.
\end{proof}

{\bf Step 4:} Using Step~3, for $\ell^\mu$-almost every $a$,  we can choose $n$ 
such that $g_n(a)\in G_n$ and
$$\sfrac{5M}{n} \leq \min\{f(y) \colon y\in[a+\eps,\tau_*(a)-\eps]\}.$$

Recall that $g_{n}(a)=a_i$, if $a \in (a_i, a_{i-1}]$. In this case we also write $g^{\ssup k}_{n}(a)=a_{i-k}$.
\smallskip

By  Step~1 and choice of $n$ we have
$$f_n(x) \geq f(x)- \sfrac{4M}{n}  \geq \sfrac{M}{n}  \mbox{ for all } x\in \big[a+\eps, \tau_*(a)-\eps \big] .$$
By Step~2, using again Step~1, 
$$f_n(x) \geq f(x)- \sfrac{4M}{n}  \geq \sfrac56 \ell^\mu[a, g^{\ssup 7}_n(a)]- \sfrac{4M}{n}
 \geq \sfrac{M}{n} 
 \mbox{ for all } x\in \big[ g^{\ssup 7}_n(a) ,a+\eps \big] ,$$
and the fact that $g_n(a)\in G_n$ implies that
$$f_n(x) \geq \sfrac{M}{n}  \mbox{ for all } x\in \big[g_n(a), g^{\ssup 7}_n(a)\big].$$
We thus obtain that
$$\tau_n(g_{n}(a))=\inf\{x>g_n(a) \colon f_n(x)=0\} \geq \tau_*(a)-\eps.$$

{\bf Step 5:} For $\ell^\mu$-almost every $a$, we have, for sufficiently large~$n$,
$$\tau_n(g_{n}(a))\leq \tau_*(a)+\eps.$$

\begin{proof}
Recall that $h_{n}(b)=b_j$, if $b \in (b_{j-1}, b_{j}]$. In this case we also write $h^{\ssup k}_{n}(b)=b_{j+k}$.
We note from Step~1 that \smash{$f_n(\tau_*(a)) \leq \frac{4M}{n}$}. Let 
$G'_n := \{ b_j \colon j\in\{1,\ldots,n\} ,  \not\!\!\exists a_i \in [b_j, b_{j+5}] \}.$
Then, as in Step~3, for $\ell^\mu$-almost every $a$, there exists arbitrarily large $n$ with 
$h_n(\tau_*(a))\in G'_n$. We infer that
$\tau_n(g_{n}(a))\leq h_n^{\ssup 5}(\tau_*(a))
\leq \tau_*(a)+\eps,$
if $n$ is large enough.
\end{proof}
The result follows by combining Steps~4 and~5, as $\eps>0$ was arbitrarily small.
\end{proof}

To conclude we note that, using Lemmas~\ref{L:g_h_convergences} and~\ref{L:tau_n convergence},
\begin{align*}
\lim_{n\to\infty} \frac{M}{n} \sum_{i=1}^{n}\psi(\tau_n(a_i)-a_{i}) &
=\lim_{n\to\infty} \int_{\mathcal E} \psi(\tau_n(g_n(a))-g_n(a)) \, \ell^\mu(da)\\
& = \int_{\mathcal E} \lim_{n\to\infty} \psi(\tau_n(g_n(a))-g_n(a)) \, \ell^\mu(da)
= \int \psi(\tau_*(a)-a) \, \ell^\mu(da),
\end{align*}
by bounded convergence.

\section{Ergodicity}\label{S:ergodicity}

Denote by $\p^{\ssup \mu}=\int \p_x\, \mu(dx)$ the law of two-sided Brownian motion 
with $\omega_0\sim \mu$ or, in other words, the push-forward of $\prob_\mu$ under  
$(B_t \colon t\in \R)$. This is a measure on the path space $(\Omega, \mathscr{A})$ introduced 
at the beginning of Section~2.
For $r\in\R$ define $S^{r}$ to be the generalized inverse of the local time  $\ell^{\mu}$, that is 
\begin{equation}
S^{r}:=\left\{
                \begin{array}{ll}
                  \sup\{t \ge 0: \ell^{\mu}[0,t]= r\}, \qquad & \text{if} \quad r\ge0, \\
                  \sup\{t < 0: \ell^{\mu}[t,0]= -r\}, \qquad  & \text{if} \quad r < 0.\\
                \end{array}
              \right.
\end{equation}
\ \\[-1mm]
Then $\p^{\ssup \mu}$ is invariant under the shifts $S^{r}\colon (\omega_s) \mapsto 
(\omega_{S^r+s})$, for every $r \in \R$,  by Theorem~3.4 in~\cite{LMT14}. In order to show that  the family of shifts $(S^{r})$ is ergodic, we need to show that any invariant set~$A$ is trivial,  i.e. $\p^{\ssup \mu}(A) \in \{0,1\}$. We follow a classical approximation approach. 
By~\cite[Appendix A.3]{Dur} there exist sets 
$A_r \in \sigma(\omega_s \colon s\in [S^{-r},S^{r}])$, $r>0$, such that 
$\p^{\ssup \mu}(A \Delta A_r) \to 0$ as $r \rightarrow \infty,$
where $\Delta$ denotes the symmetric difference of the two sets. 
As $A$ is invariant under the shift $S^{2r}$ we~get
\begin{align}
\p^{\ssup \mu}(A\Delta S^{2r}A_r)=\p^{\ssup \mu}(S^{2r}A \Delta S^{2r}A_r)=\p^{\ssup \mu}(S^{2r}(A\Delta A_r))=\p^{\ssup \mu}(A\Delta A_r) \to 0. \notag 
\end{align}
Hence there exists $r_n\nearrow\infty$ such that
$$\sum_{n=1}^\infty \p^{\ssup \mu}(A\Delta S^{2r_n}A_{r_n})<\infty,$$
and 
$A\setminus \bigcap_{s>0} \bigcup_{r_n>s} S^{2r_n}A_{r_n} $ and
$\bigcap_{s>0} \bigcup_{r_n>s} S^{2r_n}A_{r_n} \setminus A$ are $\p^{\ssup \mu}$-nullsets.
This implies
$$\p^{\ssup \mu}(A)= \p^{\ssup \mu}\Big( \bigcap_{s>0} \bigcup_{r_n>s} S^{2r_n}A_{r_n}\Big)\in\{0,1\},$$
using that the latter event is a tail event and hence trivial,  see e.g.\  \cite[Theorem 2.9]{MP10}.
\smallskip

\begin{lemma}\label{L:ergodic}
Let $T \ge 0$ be an unbiased shift and $\theta$ be the associated transport rule.
Furthermore let $\psi\colon [0, \infty) \to [0,\infty)$ be concave with $\psi(0)=0$. 
Let $S_{\mu+\nu}^{r}$ 
be  the generalized inverse of the local time~$\ell^{\mu+\nu}$.
Then, $\p^{\ssup{\mu+\nu}}$-almost surely,
\[
\lim_{r \to \infty}\frac{1}{r} \iint_{0}^{S_{\mu+\nu}^{r}}\psi(t-s)\, \theta(ds,dt) = 
\lim_{r \to \infty}\frac{1}{r}\iint_{S_{\mu+\nu}^{-r}}^{0}\psi(t-s)\, \theta(ds,dt)
= \sfrac12\, \E_{\mu}^{}\psi(T),
\]
and, 
\[ 
\lim_{r \to \infty}\frac{1}{r} \int_{0}^{S_{\mu+\nu}^{r}} \int \psi(t-s)\, \theta(ds, dt)
=\lim_{r \to \infty}\frac{1}{r} \int_{S_{\mu+\nu}^{-r}}^{0} \int \psi(t-s)\, \theta(ds,dt)
= \sfrac12\,  \E_{\mu}^{}\psi(T).
\]
\end{lemma}
\begin{proof}
By the ergodic theorem, $\p^{\ssup{\mu+\nu}}$-almost surely,
\begin{align}
\lim_{r \to \infty}\frac{1}{r} \iint_{0}^{S_{\mu+\nu}^{r}}\psi(t-s)\, \theta(ds,dt) & =
\lim_{r \to \infty}\frac{1}{r}\int_{0}^{S_{\mu+\nu}^{r}} \bigg(\int\psi(t-s)\, \theta(s,dt) \bigg)\, \ell^{\mu+\nu}(ds)\notag\\ 
& = \lim_{r \to \infty}\frac{1}{r}\int_{0}^{{r}} \bigg( \int \psi(t-S_{\mu+\nu}^s)\, \theta(S_{\mu+\nu}^s,dt) \bigg) \, ds\notag\\
& = \sfrac12\, \E^{\ssup{\mu+\nu}}\int \psi(t) \, \theta(0,dt)=\sfrac12\, \E_{\mu}^{}\psi(T). \notag
\end{align}
Similarly,  we obtain $\p^{\ssup{\mu+\nu}}$-almost surely,
\begin{align}
\lim_{r \to \infty}\frac{1}{r}\int_{0}^{S_{\mu+\nu}^{r}} \int \psi(t-s)\,\theta(ds,dt)
& =\sfrac12\, \E^{\ssup{\mu+\nu}}\int \psi(-s) \theta(ds,0) \notag\\
& =\sfrac12\, \E^{\ssup \nu}\int \psi(-s) \, \theta(ds,0).
\notag
\end{align}
Using the generalized Campbell formula, see \cite[(2.4)]{LMT14}, we get 
\begin{align*}
\E^{\ssup \nu}\int \psi(-s) \, \theta(ds,0)
& = \E^{\ssup \nu}\int  \bigg( \int \psi(-s) \1\{s+r\in[0,1]\}\, \theta_\omega(ds,0) \bigg) \, dr \\
& = \E \int  \bigg( \int \psi(-s) \1\{s+r\in[0,1]\}\, \theta_{r\omega}(ds,0) \bigg) \, \ell^\nu(dr)\\
& = \E \iint \1\{t\in[0,1]\} \psi(r-t) \,  \theta_\omega(dt, r) \ell^\nu(dr),
\end{align*}
where in the last equation we used the shift-invariance of $\theta$. Using the balancing property 
first and then the generalized Campbell formula again this equals
\begin{align*}
\E \iint \1\{t\in[0,1]\} \psi(r-t) \,  \theta_\omega(t, dr) \ell^\mu(dt)
& = \E^{\ssup \mu}\int \psi(t) \, \theta(0,dt)= \E_{\mu}^{}\psi(T).
\end{align*}
The claims about backward time  follow in the same manner.
\end{proof}

\begin{proof}[Proof of Theorem~\ref{T:Brownian_optimality}]
Define for $u>0$  the stopping times
\begin{align}
\rho(u)&=\inf\{t \ge 0: \ell^{\mu}([0,t])-\ell^{\nu}([0,t])=-u\},\notag \\
\sigma(u)&=\sup\{t \le 0: -\ell^{\mu}([t,0])+\ell^{\nu}([t,0])=-u\}.\notag
\end{align}
We have
$\rho(u)\nearrow \infty$ and $\sigma(u)\searrow -\infty$, as $u \to \infty$, and, for all $u>0$, the interval $[\sigma(u),\rho(u)]$ forms an excursion.  
Hence, by Proposition~\ref{L:pathwise},
\begin{align}
 \int_{\sigma(u)}^{\rho(u)}\int \psi(t-s) \, \theta(ds,dt)
+ \int \int_{\sigma(u)}^{\rho(u)}\psi(t-s)\, \theta(ds,dt)
\geq 2 \int_{\sigma(u)}^{\rho(u)}\psi(\tau_{*}(s)-s) \, \ell^{\mu}(ds). \notag
\end{align}
Applying Lemma~\ref{L:ergodic} the left hand side is asymptotically equal to 
$\frac12\, \ell^{\mu+\nu}([\sigma(u),\rho(u)])\E_{\mu}\psi(T),$
and the right hand side to $\frac12\, \ell^{\mu+\nu}([\sigma(u),\rho(u)])\E_{\mu}\psi(T_{*})$, which concludes the proof.
\end{proof}
\bigskip

{\bf Acknowledgements:} 
This paper is part of the second author's PhD project. We would like to thank the referee for their careful reading of the manuscript.
\\

\bigskip

Peter M\"orters and Istv\'an Redl\\
Department of Mathematical Sciences\\
University of Bath\\
Bath BA2 7AY\\
United Kingdom\\


\bigskip

\begin{thebibliography}{88}

\bibitem[1]{BH13}
M.~Beiglb\"ock, A.~M.~G. Cox, and M.~Huesmann.
\newblock Optimal transport and {S}korokhod embedding.
\newblock {\em Inventiones Math.}, to appear, 2016.

\bibitem[2]{BL92}
J.~Bertoin and Y.~Le~Jan.
\newblock Representation of measures by balayage from a regular recurrent
  point.
\newblock {\em Ann. Probab.}, 20:538--548, 1992.

\bibitem[3]{CH04}
A.~M.~G. Cox and D.~G. Hobson.
\newblock An optimal {S}korokhod embedding for diffusions.
\newblock {\em Stochastic Process. Appl.}, 111:17--39, 2004.

\bibitem[4]{CH07}
A.~M.~G. Cox and D.~G. Hobson.
\newblock A unifying class of {S}korokhod embeddings: connecting the
  {A}z\'ema-{Y}or and {V}allois embeddings.
\newblock {\em Bernoulli}, 13:114--130, 2007.

\bibitem[5]{Dur}
R.~Durrett.
\newblock {\em Probability: theory and examples}.
\newblock Cambridge Series in Statistical and Probabilistic Mathematics.
  Cambridge University Press, Cambridge, fourth edition, 2010.

\bibitem[6]{EG}
L.~C. Evans and R.~F. Gariepy.
\newblock {\em Measure theory and fine properties of functions}.
\newblock Studies in Advanced Mathematics. CRC Press, Boca Raton, FL, 1992.

\bibitem[7]{HL01}
A.~E. Holroyd and T.~M. Liggett.
\newblock How to find an extra head: optimal random shifts of {B}ernoulli and
  {P}oisson random fields.
\newblock {\em Ann. Probab.}, 29:1405--1425, 2001.

\bibitem[8]{HPPS09}
A.~E. Holroyd, R.~Pemantle, Y.~Peres, and O.~Schramm.
\newblock Poisson matching.
\newblock {\em Ann. Inst. Henri Poincar\'e Probab. Stat.}, 45:266--287, 2009.

\bibitem[9]{HP05}
A.~E. Holroyd and Y.~Peres.
\newblock Extra heads and invariant allocations.
\newblock {\em Ann. Probab.}, 33:31--52, 2005.

\bibitem[10]{LMT14}
G.~Last, P.~M{\"o}rters, and H.~Thorisson.
\newblock Unbiased shifts of {B}rownian motion.
\newblock {\em Ann. Probab.}, 42:431--463, 2014.

\bibitem[11]{LT09}
G.~Last and H.~Thorisson.
\newblock Invariant transports of stationary random measures and
  mass-stationarity.
\newblock {\em Ann. Probab.}, 37:790--813, 2009.

\bibitem[12]{LT02}
T.~M. Liggett.
\newblock Tagged particle distributions or how to choose a head at random.
\newblock In {\em In and out of equilibrium ({M}ambucaba, 2000)}, volume~51 of
  {\em Progr. Probab.}, pages 133--162. Birkh\"auser, Boston, MA, 2002.

\bibitem[13]{MP10}
P.~M{\"o}rters and Y.~Peres.
\newblock {\em Brownian motion}.
\newblock Cambridge Series in Statistical and Probabilistic Mathematics.
  Cambridge University Press, Cambridge, 2010.

\bibitem[14]{MR15}
P.~M{\"o}rters and I.~Redl.
\newblock Skorokhod embeddings for two-sided {M}arkov chains.
\newblock {\em Probability Theory and Related Fields}, 165:483--508, 2016.

\bibitem[15]{Th00}
H.~Thorisson.
\newblock {\em Coupling, Stationarity, and Regeneration}.
\newblock Probability and its Applications (New York). Springer-Verlag, New
  York, 2000.

\end{thebibliography}
\end{document}